\documentclass[a4paper]{amsart}

\usepackage{amscd,amsthm,amsfonts,amssymb,amsmath,esint}
\usepackage[colorlinks=true]{hyperref}
\usepackage{fullpage}
\usepackage[all]{xy}

\usepackage{mathrsfs}

\newcommand{\msi}{\mathscr{I}}\newcommand{\msl}{\mathscr{L}}

\newcommand{\mst}{\mathscr{T}}

    \newcommand{\BC}{{\mathbb {C}}}

    \newcommand{\BQ}{{\mathbb {Q}}}

     \newcommand{\BZ}{{\mathbb {Z}}}

    \newcommand{\CE}{{\mathcal {E}}}

     \newcommand{\CL}{{\mathcal {L}}}
     
    \newcommand{\CO}{{\mathcal {O}}}

    \newcommand{\fa}{{\mathfrak{a}}} \newcommand{\fb}{{\mathfrak{b}}}

    \newcommand{\fg}{{\mathfrak{g}}}

     \newcommand{\fp}{{\mathfrak{p}}}
    \newcommand{\fq}{{\mathfrak{q}}} \newcommand{\fr}{{\mathfrak{r}}}

      \newcommand{\fB}{{\mathfrak{B}}}
    \newcommand{\fC}{{\mathfrak{C}}} \newcommand{\fD}{{\mathfrak{D}}}
     
    \newcommand{\fG}{{\mathfrak{G}}} 
     \newcommand{\fJ}{{\mathfrak{J}}}
     \newcommand{\fL}{{\mathfrak{L}}}
     
     \newcommand{\fP}{{\mathfrak{P}}}
    \newcommand{\fQ}{{\mathfrak{Q}}} \newcommand{\fR}{{\mathfrak{R}}}
    \newcommand{\fS}{{\mathfrak{S}}} 
     
    \newcommand{\fW}{{\mathfrak{W}}}

    \newcommand{\End}{{\mathrm{End}}}

    \newcommand{\Gal}{{\mathrm{Gal}}}

    \newcommand{\ord}{{\mathrm{ord}}}

    \renewcommand{\mod}{\ \mathrm{mod}\ }\renewcommand{\Re}{{\mathrm{Re}}}

    \font\cyr=wncyr10

    \newcommand{\Sha}{\hbox{\cyr X}}

    \newcommand{\ov}{\overline}

    \newcommand{\ra}{\rightarrow}

    \theoremstyle{plain}
    \newtheorem{thm}{Theorem}[section] \newtheorem{cor}[thm]{Corollary}
    \newtheorem{lem}[thm]{Lemma}  \newtheorem{prop}[thm]{Proposition}

\theoremstyle{remark} 
\theoremstyle{remark} 
\theoremstyle{remark} 

    \numberwithin{equation}{section}

\begin{document}

\title{Non-vanishing theorems for central $L$-values of some elliptic curves with complex multiplication II.}

\author{John Coates, Yongxiong Li}

\begin{abstract} Let $q$ be any prime $\equiv 7 \mod 16$, $K = \mathbb{Q}(\sqrt{-q})$, and let $H$ be the Hilbert class field of $K$.  Let $A/H$ be the Gross elliptic curve defined over $H$ with complex multiplication by the ring of integers of $K$. We prove the existence
of a large explicit infinite family of quadratic twists of $A$ whose complex $L$-series does not vanish at $s=1$. This non-vanishing
theorem is completely new when $q > 7$. Its proof depends crucially on the results established in our earlier paper for the Iwasawa theory at the prime
$p=2$ of the abelian variety $B/K$, which is the restriction of scalars from $H$ to $K$ of the elliptic curve $A$.
\end{abstract}
\maketitle
\section{Introduction}

 Let $K=\BQ(\sqrt{-q})$, where $q$ is any  prime with $q \equiv 7 \mod 8$. We fix an embedding of $K$ into the field $\BC$ of complex numbers. Let $\CO_K$ be the ring of integers of $K$, and write $h$ for the class number of $K$. Let $H=K(j(\CO_K))$ be the Hilbert class field of $K$, where $j$ is the classical elliptic modular function. Gross \cite{Gross78} has shown that there exists a unique elliptic curve $A$ defined over $\BQ(j(\CO_K))$, with complex multiplication by $\CO_K$, minimal discriminant $(-q^3)$, and which is a $\BQ$-curve in the sense that it is isogenous over $H$ to all of its conjugates.  An explicit  equation for $A$ over $\BQ(j(\CO_K))$ is given by
\begin{equation}\label{mg}
y^2 = x^3 + 2^{-4}3^{-1}mqx - 2^{-5}3^{-3}rq^2,
\end{equation}
where $m^3 = j(\CO_K)$, and $r^2 = ((12)^3 - j(\CO_K))/q$ with $r > 0$ (see \cite{Gross82}). Let $L(A/H, s)$ be the complex $L$-series of $A/H$. When $q \equiv 7 \mod 16$, we used Iwasawa theory in \cite{CL} to give a new proof of an old theorem of Rohrlich \cite{Ro1} asserting that $L(A/H, 1) \neq 0$. The aim of the present paper is to prove a generalization of  this theorem to a large infinite class of quadratic twists of $A/H$. Let $\fR$ denote the set of all square free positive integers $R$ of the form $R = r_1...r_k$ , where $k\geq 0$,  and $r_1, \dots, r_k$ are distinct primes such that (i) $r_i \equiv 1 \mod 4$, and  (ii) $r_i$ is inert in $K$, for $i=1,..., k$. For $R \neq 1 \in \fR$, let $A^{(R)}$ be the twist of $A$ by the quadratic extension $H(\sqrt{R})/H$. Note that such an extension is non-trivial since $K$ has odd class number. We write $L(A^{(R)}/H, s)$ for the complex $L$-series of $A^{(R)}/H$. By Deuring's theorem, $L(A^{(R)}/H, s)$ is a product of Hecke $L$-series with Grossencharacter. We shall prove the following theorem.

\begin{thm}\label{main}
Assume that  $q\equiv 7\mod 16$. Then, for all $R \in \fR$, we have $L(A^{(R)}/H, 1) \neq 0$.
\end{thm}

\noindent Since the torsion subgroup of $A^{(R)}(H)$ is equal to $\CO_K/2\CO_K$, it has been shown independently by several authors (J. Choi \cite{JC}, and K. Li and Y. Ren \cite{LR}) that one can carry out a classical 2-descent argument on $A^{(R)}/H$ and prove that $A^{(R)}(H)$ is in fact finite for all $R \in \fR$. However, such an argument tells us nothing about the finiteness of the 2-primary subgroup of the Tate-Shafarevich group $\Sha(A^{(R)})$ for $R \in \fR$, and thus does not allow any hope of proving  a result like Theorem \ref{main} easily from a suitable main conjecture for $A^{(R)}$ over an appropriate $\BZ_2$-extension of $H$.  We remark that for the special case $q=7$, the curve $A$ is the modular elliptic curve $X_0(49)$, and the above type theorem is already proved in \cite{CLTZ} by two rather different methods. In fact, we use a modification of one of these methods, due originally to C. Zhao \cite{Zhao1}, \cite{Zhao2}, to prove the above theorem.   However, the proof is considerably more difficult when $q > 7$. In particular, we have to make essential use of the abelian variety $B/K$ which is the restriction of scalars from $H$ to $K$ of $A/H$, and appeal to one of the main results of our earlier paper \cite{CL}, to establish several key results  (see Theorem \ref{key-identity} and
Proposition \ref{4.4n}) needed for carrying out Zhao's argument in this case. In addition, we need a rather delicate result from \cite{BG}.  In a subsequent paper with Y. Kezuka and Y. Tian \cite{CKLT}, we shall use the methods of Iwasawa theory to show that the above theorem  implies that the Tate-Shafarevich group of $A^{(R)}/H$ is indeed finite for all $R \in \fR$ when $q \equiv 7 \mod 16$, and that the order of its $p$-primary subgroup is as predicted by the conjecture of Birch and Swinnerton-Dyer for all primes $p$ which split in $K$. We are even optimistic that one should eventually be able to prove the full Birch-Swinnerton-Dyer conjectural formula for the the order of the Tate-Shafrevich group of $A^{(R)}/H$ for all $R \in \fR$ .

\section{ Preliminaries}

Gross \cite{Gross82}has shown that there always exists a global minimal Weierstrass equation for $A/H$. We fix one such equation for the rest of the paper
\begin{equation}\label{1}
  y^2+a_1xy+a_3y=x^3+a_2x^2+a_4x+a_6
\end{equation}
whose coefficents $a_i$ are integers in $H$.  We write $B/K$ be the abelian variety which is the restriction of scalars from $H$ to $K$ of the Gross curve $A/H$. For each $R \in \fR$, we write $B^{(R)}$ for the twist of $B$ by the quadratic extension $K(\sqrt{R})/K$, and $A^{(R)}$ for the twist of $A$ by the quadratic extension $H(\sqrt{R})/H$. It is easily seen that $B^{(R)}$ is in fact the restriction of scalars from $H$ to $K$ of $A^{(R)}$. Let $\psi$ for the Serre-Tate character of $A/H$, and $\phi$ the Serre-Tate character of $B/K$, so that $\psi = \phi \circ N_{H/K}$, where $N_{H/K}$
denotes the norm map from $H$ to $K$. Then, since $(R, q) = 1$, the Serre-Tate character  $\phi_R$ of $B^{(R)}/K$ is given by $\phi_R=\phi\chi_R$, where $\chi_R$ denotes the abelian character of $K$ defining the quadratic extension $K(\sqrt{R})/K$. Moreover, since $H/K$ is unramified,  the Serre-Tate character  $\psi_R$ of $A^{(R)}/H$ is then equal to
$\psi_R=\phi_R\circ N_{H/K}.$

 We write
$$
\mst=\End_K(B^{(R)})=\End_K(B), \, \, T = \mst\otimes_\BZ \BQ.
$$
Then $T$ is a CM field of degree $h$ over $K$, where $h$ denotes the class number of $K$.  In what follows, we shall write $\ov{\phi}$ for the Grossencharacter obtained by applying the unique complex conjugation in $T$ to $\phi$. Moreover, the index of $\mst$ in the maximal order of $T$ is prime to 2 (see \cite{Gross78}, \S13).  As $q\equiv 7\mod 8$, the prime $2$ splits in $K$ into two distinct primes, which we will denote by $\fp$, and $\fp^*$. The following lemma (see \cite{BG} or Lemma 2.1 in \cite{CL}) gives the existence of a degree one prime of $T$ above $\fp$, which will play a fundamental role in our subsequent  generalization of Zhao's induction argument.

\begin{lem}\label{s1}
There exists an unramified degree one prime $\fP$ of $T$ lying above $\fp$.
\end{lem}

\noindent Of course, since the index of $\mst$ in the maximal ideal of $T$ is prime to 2, $\fP \cap \mst$ will be a degree one prime ideal of $\mst$, which for simplicity we shall again denote by $\fP$.

\section{A special value formula for certain Hecke characters}

Throughout $R$ will denote an arbitrary positive integer in the set $\fR$. For each positive divisor $d>1$ of $R$, we let $\chi_d$ be the non-trivial character of $\Gal(K(\sqrt{d})/K)$, and we define $\phi_d = \phi\chi_d$, where, as always, $\phi$ denotes the Hecke character attached to the abelian variety $B/K$. We also put $\phi_1 = \phi$. Recall that  $\fq = \sqrt{-q}\cdot\CO_K$ is the conductor of $\phi$.

\begin{lem}\label{3.1}
For every positive divisor $d$ of $R$, the Hecke character $\phi_d$ has conductor $d\fq$.
\end{lem}

\begin{proof}
The positive integer $d$ is square free, and satisfies $d \equiv 1 \mod 4$ since each prime dividing $d$ is $\equiv 1 \mod 4$, from which it follows that $\chi_d$ has conductor $d\CO_K$. But $\phi$ has conductor $\fq$, which is prime to $d\CO_K$, and thus $\phi_d$ will have conductor $d\fq$.
\end{proof}

The Neron differential $\omega$ of $A$ is given by $\omega=dx/(2y+a_1x+a_3)$. Now we have fixed an embedding of $K$ into $\BC$, and, since $H = K(j(\CO_K)$,
this extends to a fixed embedding of $H$ into $\BC$. Thus the complex period lattice $\CL$ of $\omega$ under this embedding must be of the form $\CL=\Omega_\infty\CO_K$, where $\Omega_\infty$ is a non-zero complex number which is uniquely determined up to sign. We also need to consider the conjugate curves of $A/H$ under the action by the Galois group $\fG=\Gal(H/K)$. If $\fa$ is any integral ideal of $K$ prime to $R\fq$, let $\gamma_\fa$ be the Artin symbol of $\fa$ in $\fG$, and write $A^\fa/H$
for the conjugate of $A/H$ under the action of $\gamma_\fa$. Clearly a global minimal Weierstrass equation for $A^\fa/H$ is given by applying $\gamma_\fa$
to the coefficients of the equation \eqref{1}, and the Neron differential of this equation is $\omega_\fa = dx/(2y+\gamma_\fa(a_1)x + \gamma_\fa(a_3) )$. We can describe the complex period lattice $\fL_\fa$ of $\omega_\fa$ as follows. Since $B/H$ is isomorphic to $\prod_{\gamma_\fa\in\fG}A^\fa$, then, for any integral ideal $\fa$ of $K$ prime to $\fq$, the endomoprhism   $\phi(\fa)$ of $B$, when restricted to $A$, gives an $H$-isogeny
$\eta_{A}(\fa): A \ra A^{\fa}$ with kernel $A_\fa$. Taking the pull back of the differential $\omega_\fa$ under this isogeny, we conclude that there exists a non-zero $\xi(\fa)$ in $H$ such that $\eta_A(\fa)^*(\omega_\fa)=\xi(\fa)\omega.$ Thus the period lattice $\CL_\fa$ of $\omega_\fa$ is given by $\CL_\fa=\xi(\fa)\fa^{-1}\Omega$.

For each positive divisor $d$ of $R$, write $E =A^{(d)}$ for the twist of the Gross curve $A/H$ by the extension $H(\sqrt{d})/H$.  When $\fb$ is any non-zero integral ideal of $\CO_K$, let $E_\fb$ be the set of all elements of $E(\ov{K})$ which are annihilated by every endomorphism in $\fb$. The following lemma is very well known (see, for example,
Lemma 4.1 of \cite{CL}).

\begin{lem}\label{3.2}
For every positive divisor $d$ of $R$, $H(E_{R\fq })$ is the ray class field over $K$ modulo $R\fq$.
\end{lem}

\noindent Note also that for the curve $E$, and each integral ideal $\fa$ of $K$ which is prime to $R\fq$, the endomorphism $\phi_d(\fa)$ of $B^{(d)}$ defines a $H$-isogeny $\eta_{E}(\fa): E \ra E^{\fa}$ with kernel $E_\fa$. Gross has shown in \cite{Gross82} (see Prop. 4.3) that the differential $\omega/\sqrt{d}$ on $A/H(\sqrt{d})$ descends to a global minimal differential on $E/H$, which we denote by $\omega(d)$. Again we write $\omega(d)_\fa = \gamma_\fa(\omega(d))$ for the Neron differential on the curve $E^\fa$. The following lemma is then clear from Gross' result, and Proposition 4.10, (vi) of \cite{GS}.

\begin{lem}\label{min} For each positive divisor $d$ of $R$, the complex period lattice of $\omega(d)$ is equal to $\Omega_\infty\CO_K/\sqrt{d}$, and, for each integral ideal $\fa$ of $K$ prime to $R\fq$, we have $\eta_E(\fa)^*(\omega(d)_\fa)=\xi_d(\fa)\omega(d)$, where $\xi_d(\fa) = \xi(\fa)/\chi_d(\fa)$.
\end{lem}

 Always assuming that $d$ is a positive divisor of $R$, we define the imprimitive partial Hecke $L$-series
\begin{equation}\label{2}
L_R(\ov{\phi_d},\gamma_\fa, s)=\sum_{ (\fb,R\fq)=1, \gamma_\fb=\gamma_\fa}\frac{\ov{\phi_d}(\fb)}{N(\fb)^s}.
\end{equation}
where the sum on the right is taken over all integral ideals $\fb$ of $K$, which are prime to $R\fq$, and are such that $\gamma_\fb=\gamma_\fa$. It is classical
that the Dirichlet series on the right converges for $R(s) > 3/2$, and it has a holomorphic continuation to the whole complex plane. We first recall a classical formula for $L_R(\ov{\phi_d},\gamma_\fa, 1)$, which essentially goes back to the 19th century (see \cite{GS}). If $\fL$ be any lattice in the complex plane, recall that the Kronecker-Eisenstein series $H_1(z,s,\fL)$ is defined by
\begin{equation}\label{3}
H_1(z,s,\fL)=\sum_{w\in \fL} \frac{\ov{z+w}}{|z+w|^{2s}},
\end{equation}
where the sum is taken over all $w\in \fL$, except $-z$ if $z\in \fL$. This series converges in the half plane $\Re(s)>\frac{3}{2}$, and it has an analytic continuation to the whole $s$-plane. We then define the Eisenstein series $\CE^*_1(z,\fL)$ of weight 1 by
$$
\CE^*_1(z,\fL)=H_1(z,1,\fL).
$$
 We write $K(R\fq)$ for the ray class field of $K$ modulo $R\fq$, so that, by Lemma \ref{3.2},
we have $H(E^\fa_{R\fq}) = K(R\fq),$ for all integral ideals $\fa$ of $K$ which are prime to $R\fq$. Let $\textrm{Tr}_{K(R\fq)/H}$ be the trace map from $K(R\fq)$ to $H$.
\begin{prop}\label{3.3n}
Assume $R \in \fR$, and let $d$ be any positive divisor of $R$. Then, for all integral ideals $\fa$ of $K$, which are prime to $R\fq$, we have
\begin{equation}\label{4}
\frac{\phi_d(\fa)R\sqrt{-qd}}{\xi_d(\fa)}\cdot \frac{L_R(\ov{\phi}_d,\gamma_\fa,1)}{\Omega_\infty}=\textrm{Tr}_{K(R\fq)/H}\left(\CE^*_1\left(\frac{\xi_d(\fa)\Omega_\infty}{R\sqrt{-qd}},\frac{1}{\sqrt{d}}\CL_\fa\right)\right).
\end{equation}
\end{prop}

\begin{proof} We apply Proposition 5.5 of \cite{GS} to the curve $E=A^{(d)}$ over $H$, with $\fg = R\fq$, and  $ \rho = \Omega_\infty/(R\sqrt{-qd})$. Recalling that $H(E_\fg) = K(Rq)$ by Lemma \ref{3.2}, and that  $\eta_E(\fa)^*(\omega(d)_\fa)=\xi_d(\fa)\omega(d)$ by Lemma \ref{min}, we conclude that
\begin{equation}\label{5}
\frac{\phi_d(\fa)R\sqrt{-qd}}{\xi_d(\fa)}\cdot \frac{L_R(\ov{\phi}_d,\gamma_\fa,1)}{\Omega_\infty} = \sum_{\fb \in \fB} \CE^*_1\left(\frac{\phi_d(\fb)\xi_d(\fa)\Omega_\infty}{R\sqrt{-qd}},\frac{1}{\sqrt{d}}\CL_\fa\right),
\end{equation}
where $\fB$ denotes any set of integral prime ideals of $K$, prime to $R\fq$, whose Artin symbols in $\Gal(K(R\fq)/K)$ give precisely $\Gal(K(R\fq)/H)$.   However, by \cite{GS},
Theorem 6.2, we note that, since the Artin symbol of every ideal $\fb \in \fB$ fixes the field of definition $H$ of $E$, the right hand side of \eqref{5} is none other than the right hand side of \eqref{4}. This completes the proof.
\end{proof}

As in the Introduction, we write $R = r_1...r_k$, where $k \geq 1$, and the $r_i$ are distinct rational primes. Define the two fields
\begin{equation}\label{6}
J_R = K(\sqrt{r_1}, \dots, \sqrt{r_k}), \, \, H_R = H(\sqrt{r_1}, \dots, \sqrt{r_k}).
\end{equation}
\begin{lem}\label{3.4n} We have $J_R \cap H = K$, $[H_R:J_R] = h$, and $H_R \subset K(R\fq)$. Moreover, for each positive divisor $d$ of $R$, $B^{(d)}$ is isomorphic to
$B$ over $J_R$, and $A^{(d)}$ is isomorphic to $A$ over $H_R$.

\end{lem}
\begin{proof} Since the class number $h$ is odd, and $[J_R:K] = 2^k$ by Kummer theory, it follows that $J_R \cap H = K$. Moreover, the extension $K(\sqrt{r_i})/K$ has
conductor $r_i\CO_K$ since $r_i \equiv 1 \mod 4$, and thus this extension is a subfield of $K(R\fq)$. The final assertion of the lemma is also clear since $J_R$ contains $\sqrt{d}$.
\end{proof}

It follows from \eqref{5} that, for each divisor $d$ of $R$, the partial $L$-value $\sqrt{d}L_R(\ov{\phi}_d,\gamma_\fa,1)/\Omega_\infty$ belongs to the compositum of fields $HT$.  Recall that $\fP$ is a degree one ideal of $T$ above $\fp$. We fix a prime $w$ of $H_RT$ lying above the prime $\fP$ of $T$. Assume now that $R$ is fixed. For each $n \geq 0$, we let
$\fC_n$ be a set of integral ideals of $K$, prime to $R\fq$, whose Artin symbols in $\Gal(H_R(A_{\fp^{n+2}})/K))$ give precisely $\Gal(H_R(A_{\fp^{n+2}})/J_R(B_{\fP^{n+2}}))$; here, for each integer $m \geq 1$, $B_{\fP^m}$ (resp. $B^{(d)}_{\fP^m}$) denotes the Galois module of $\fP^m$-division points on the abelian variety $B$ (resp. $B^{(d)}$).
One sees easily that, for each $n \geq 0$, $\fC_n$ gives a complete set of representatives of the ideal class group of $K$. Moreover, since $J_R(B_{\fP^{n+2}}) = J_R(B^{(d)}_{\fP^{n+2}})$, we conclude that
\begin{equation}\label{7}
\phi_d(\fa) \equiv 1 \mod \fP^{n+2} \, \, \, \, {\rm for \, all}  \, \, \fa \in \fC_n.
\end{equation}
\noindent Note that, for any lattice $\fL$ and any $\lambda\neq 0\in\BC$, we have $\CE^*_1(\lambda z,\lambda \fL)=\lambda^{-1}\CE^*_1(z,\fL).$
Hence, summing the formula \eqref{5} over all $\fa\in \fC_n$, and taking $\lambda = 1/ (\sqrt{d}\chi_d(\fa))$ and $\fL = \chi_d(\fa)\CL_\fa = \CL_\fa$, we immediately obtain the equation
\begin{equation}\label{8}
  \sum_{\fa\in \fC_n}\phi_d(\fa)L_R(\ov{\phi}_d,\gamma_\fa,1)/\Omega_\infty=\sum_{\fa\in\fC_n} \xi(\fa)\sum_{\sigma\in\Gal(K(\fq R)/H)}\left(\sqrt{d}\right)^{\sigma-1}\frac{1}{R\sqrt{-q}}\CE^*_1\left(\frac{\xi(\fa)\Omega_\infty}{R\sqrt{-q}},\CL_\fa\right)^\sigma.
\end{equation}
Now the values $L_R(\ov{\phi}_d,\gamma_\fa,1)/\Omega_\infty$ are independent of $n$ since $\fC_n$ is a complete set of representatives of the ideal class group of K. Thus, we conclude from \eqref{7} that the left hand side of \eqref{8} converges $w$-adically as $n \to \infty$ to $ L_R(\ov{\phi}_d,1)/\Omega_\infty$, assuming $R$ is fixed. Therefore, the right hand side of \eqref{8}  also converges $w$-adically as $n\to \infty$, and so we have proven the following result.

\begin{prop}\label{special-val}
For every positive integer divisor $d$ of $R$, we have
\begin{equation}\label{9}
 L_R(\ov{\phi}_d,1)/\Omega_\infty =\lim_{n\ra \infty}\sum_{\fa\in\fC_n}\xi(\fa)\sum_{\sigma\in\Gal(K(R\fq)/H)}\left(\sqrt{d}\right)^{\sigma-1}\frac{1}{R\sqrt{-q}}\CE^*_1\left(\frac{\xi(\fa)\Omega_\infty}{R\sqrt{-q}},\CL_\fa\right)^\sigma.
\end{equation}.
\end{prop}
\noindent One of the key idea in Zhao's induction method is to sum the formula \eqref{9} over all positive integer divisors $d$ of $R$, and make use of the following well known lemma.
\begin{lem}
 Recall that $R = r_1 \dots r_k$, where the $r_i$ are distinct prime numbers. Let  $\sigma$ be any element of $\Gal(K(R\fq)/H)$. Then, letting $d$ run over all positive integer divisors of $R$, the expression  $\sum_{d\mid R}(\sqrt{d})^{\sigma-1}$ is equal to $2^k$ if $\sigma \in \Gal(K(R\fq)/H_R)$, and is equal to $0$ otherwise.
\end{lem}
\begin{proof} We quickly recall the proof. The first assertion of the lemma is clear. To prove the second assertion, suppose that $\sigma$ maps $j \geq 1$ elements of the set $\{\sqrt{r_1}, \dots, \sqrt{r_k}\}$ to minus themselves, and write $V(\sigma)$ for the subset consisting of all such elements. If $d$ is any positive integer divisor of $R$, then
$\sigma$ will fix $\sqrt{d}$ if and only if $d$ is a product of an even number of elements of $V(\sigma)$ with an arbitrary number of elements of the complement of $V(\sigma)$
in $\{\sqrt{r_1}, \dots, \sqrt{r_k}\}$. Thus the total number of $d$ such that $\sqrt{d}$ is fixed by $\sigma$ is equal to
$$
2^{k-j}((j, 0) + (j, 2) + (j, 4) + \dots) = 2^{k-1},
$$
where $(j, i)$ denotes the number of ways of choosing $i$ objects from a set of $j$ objects.  Similarly, the total number of $d$ such that $\sigma$ maps $\sqrt{d}$ to
$-\sqrt{d}$ is equal to
$$
2^{k-j}((j, 1) + (j, 3) + (j, 5) + \dots) = 2^{k-1},
$$
and the second assertion of the lemma is now clear.
\end{proof}
\noindent For each $\fa \in \fC_n$, define
\[\Psi_{\fa, R} =\textrm{Tr}_{K(R\fq)/H_R}\left(\frac{1}{R\sqrt{-q}}\cdot\CE^*_1\left(\frac{\xi(\fa)\Omega_\infty}{R\sqrt{-q}},\CL_\fa\right)\right).\]
 In view of the above lemma, Proposition \ref{special-val} can be rewritten as follows.
\begin{thm}\label{key-identity}
Letting $d$ runs over all positive integer divisors of $R$, we have
\begin{equation}\label{10}
  \sum_{d\mid R}L_R(\ov{\phi}_d,1)/\Omega_\infty=2^k \cdot \lim_{n\ra \infty}\sum_{\fa\in\fC_n} \xi(\fa)\Psi_{\fa, R}.
\end{equation}
\end{thm}
\noindent Finally, we shall need the following key integrality result (see \cite{C1}, \cite{CLTZ}).
\begin{prop}\label{inte-2}
For all $\fa\in\fC_n$, $\Psi_{\fa, R}$ is integral at all places of $H_R$ above $2$.
\end{prop}

\begin{proof}
We briefly recall the proof given in \cite{CLTZ}. Write $\fJ=H(A^\fa_\fq)$, which is also the ray class field $K(\fq)$.
Since $A^\fa$ is a relative Lubin-Tate formal group, in the sense of De Shalit \cite{DS85}, at each prime of H lying above the set of primes of $K$ dividing $R$, it is easily seen that the action of the  Galois group $\Gal(K(R\fq)/\fJ)$ on $A^\fa_R$ gives an isomorphism
\[\tau: \Gal(K(R\fq)/\fJ)\simeq \left(\CO_K/R\CO_K\right)^\times.\]
Since $q$ is prime to $R$, we can find $\alpha,\beta$ in $\CO_K$ such that $1=\alpha R+\beta \sqrt{-q}$.  We then define
\[z_1=\frac{\xi(\fa)\alpha\Omega_\infty}{\sqrt{-q}},\quad z_2=\frac{\xi(\fa)\beta\Omega_\infty}{R}.\]
and write $P_1$ and $P_2$ for the corresponding points on $A^\fa$ under the Weierstrass isomorphism.
For any $b \in H$, let $b^\fa =\gamma_\fa(b)$. Define $\epsilon$ to be the inverse image of the class $-1\mod R\CO_K$ under the isomorphism $\tau$, and let $\fS$ be the fixed field of $\epsilon$, so that the extension $K(R\fq)/\fS$ has degree $2$. Of course, $\fS$ contains $H_R$ because $-1$ is a square modulo $r_j$ for $j=1,\dots k$. Defining
$\Phi_\fa=\textrm{Tr}_{K(R\fq)/\fS}\left(\Psi_\fa\right)$, we have
\[\Phi_\fa=\frac{1}{R\sqrt{-q}}\cdot\left(\CE^*_1(z_1+z_2,\CL_\fa)+\CE^*_1(z_1-z_2,\CL_\fa)\right)\]
On the other hand, by a classical identity ( see Lemma 4.3 in \cite{CLTZ}), we have
\[\CE^*_1(z_1+z_2,\CL_\fa)+\CE^*_1(z_1-z_2,\CL_\fa)=2\CE^*_1(z,\CL)+\frac{2y(P_1)+a^\fa_1\cdot x(P_1)+a^\fa_3}{x(P_1)-x(P_2)}.\]
However, as is explained in detail in \cite{CLTZ}, each of the two terms on the right hand side of this last equation are integral at all places of $\fS$ lying above $2$, and the assertion of the lemma follows.
\end{proof}

\section{Zhao's method}

In this section, we will use Zhao's induction method to prove the following theorem. Let $R = r_1\ldots r_k$ be any element of the set $\fR$, so that $r_1,\ldots, r_k$ are distinct prime numbers which are inert in $K$ and $\equiv 1 \mod 4$.  Let  $\fP$ be the degree 2 prime of $T$ lying above $\fp$ whose existence is given by Lemma \ref{s1}, and we  write $\fD$ for any of the primes of the field $HT$ lying above $\fP$. We choose any embedding  $\iota: T \to \BC$ which extends our given embedding of $K$ into $\BC$. Since we have fixed an embedding of $H$ in $\BC$, we obtain an embedding of the field $HT$ in $\BC$ whose restriction to $T$ is the embedding $\iota$. Let $A$ be the Gross curve \eqref{mg}. As always, we write $B$ for the abelian variety over $K$ which is the restriction of scalars of $A$ from $H$ to $K$, and $\phi$ will be the Serre-Tate character of $B/K$, which we can view as a complex Grossencharacter thanks to our given embedding $\iota: T \to \BC$. For simplicity, we continue to write $\phi$, rather than $\phi^\iota$, for this complex Grossencharacter. Assuming $R \neq 1$, let $B^{(R)}$ be the twist of $B$ by the quadratic extension $K(\sqrt{R})/R$, and let $\phi_R$ be the Serre-Tate character of $B^{(R)}/K$, so that $\phi_R = \phi \chi_R$, where $\chi_R$ is the quadratic character of $K$ of the extension $K(\sqrt{R})/K$. Now it follows from Proposition \ref{3.3n} that the value $\sqrt{R}L(\ov{\phi}_R, 1)/\Omega_\infty$ belons to $HT$.

\begin{thm}\label{AR-nonzero}
Assume that $q\equiv 7\mod 16$, and let $R=r_1\cdots r_k$ be any element of the set $\fR$. Then, for each prime $\fD$ of $HT$ lying above the prime $\fP$ of $T$, we have
\begin{equation}\label{5.1}
  \ord_\fD(\sqrt{R}L(\ov{\phi}_R, 1)/\Omega_\infty)=k-1.
\end{equation}
In particular, $L(\ov{\phi}_R, 1) \neq 0$.
\end{thm}
\noindent When $R \neq 1$, we write $A^{(R)}/H$ for the twist of $A/H$ by the quadratic extension $H(\sqrt{R})/H$, and let $L(A^{(R)}/H, s)$ be its complex $L$-series. Letting $\iota$ range over all $h$ distinct embeddings of $T$ into $\BC$ extending the embedding of $K$ into $\BC$, we immediately obtain the following corollary.

\begin{cor} Assume that $q\equiv 7\mod 16$, and let $R$ be any element of $\fR$. Then $L(A^{(R)}/H, 1) \neq 0.$
\end{cor}

Prior to giving the proof of Theorem \ref{AR-nonzero}, we establish some preliminary lemmas.  Recall that $2\CO_K = \fp\fp^*$, and, as in \cite{CL}, let $F = K(B_{\fP^2})$.  Then $\fp$ and $\fq$
are the only two primes of $K$ which are ramified in the extension $F/K$. We thank Zhibin Liang for pointing out the following result to us.

\begin{lem}\label{4.1} If $q \equiv 7 \mod 16$, then $\fp^*$ is inert in $F$, and if $q \equiv 15 \mod 16$, then $\fp^*$ splits in $F$.
 \end{lem}

\begin{proof} As in \cite{CL}, we fix the sign of $\alpha = \sqrt{-q}$ so that $\ord_\fp((1-\alpha)/2) > 0$. Then, by Lemma 7.11 of \cite{CL}, we have $F = K(\sqrt{-\alpha})$. Thus
$F$ is the splitting field over $K$ of the polynomial $g(X) = X^2 + X + (\alpha + 1)/4$. Noting that
$$
(1+\alpha)(1-\alpha)/4 = (q+1)/4, \, \, \,  (1+\alpha)/2 + (1-\alpha)/2 = 1,
$$
it follows easily that $g(X)$ modulo $\fp^*$ is equal to $X^2 + X + 1$ if $q \equiv 7 \mod 16$, and it is equal to $X^2+ X$ if $q \equiv 15 \mod 16$. The assertions of the lemma
now follow easily.
\end{proof}

\begin{cor}\label{4.2n} If $q \equiv 7 \mod 16$, then $\ord_\fP(\phi(\fp^*) - 1) = 1$, and if $q \equiv 15 \mod 16$, then $\ord_\fP(\phi(\fp^*) - 1) \geq 2$. \end{cor}

\begin{proof} The prime $\fp^*$ is unramified in the extension $F/K$, and we let $\tau$ be its Artin symbol. Since $\phi$ is the Serre-Tate homomorphism for $B/K$, we have
$\tau(Q) = \phi(\fp^*)(Q)$ for all $Q$ in $B_{\fP^2}$, whence the assertion of the corollary follows from the previous lemma.
\end{proof}

\begin{lem}\label{4.3} Let $d$ be any positive divisor of $R$, and let $r$ be any prime dividing $R$ with $(r, d) = 1$. Then $\phi_d(r\CO_K) = - r$. \end{lem}

\begin{proof} Let $\fr = r\CO_K$. Noting that $-r$ is a square modulo $\fq$, the explicit formula for $\phi$ given at the beginning of \S2 of \cite{BG} shows that
$\phi(\fr) = - r$. On the other hand, since $r$ is inert in $K$, and the Galois group of $K(\sqrt{d})/\BQ$ is not cyclic, the prime $\fr$ of $K$ must split in $K(\sqrt{d})$,
and so we must have $\chi_d(\fr) = 1$. Hence $\phi_d(\fr) = -r$, as required.
\end{proof}

\noindent We first show that Theorem \ref{AR-nonzero} holds for $R = 1$.

\begin{prop}\label{4.4n} Assume that $q \equiv 7 \mod 16$. Then, for all primes $\fQ$ of $HT$ above $\fP$, we have $\ord_\fQ(L(\ov{\phi}, 1)/\Omega_\infty) = -1$.
\end{prop}

\begin{proof} The proof makes essential use of the so called "main conjecture" for $B$ over the field $F_\infty = K(B_{\fP^\infty})$, which is given
by Theorem 7.13 of \cite{CL}. Put $\Gamma = \Gal(F_\infty/F)$. Let $\msi$ be the ring of integers of the completion of the maximal unramified extension of $K_\fp$,
and write $\Lambda_\msi(\Gamma)$ for the Iwasawa algebra of $\Gamma$ with coefficients in $\msi$. Then, as is explained in the proof of Corollary 7.15 of \cite{CL},
the full force of the main conjecture tells us that the measure $\mu_A$ appearing in Theorem 7.13 of \cite{CL} is a unit in $\Lambda_\msi(\Gamma)$ when
$q \equiv 7 \mod 16$. Hence the integral of any continuous homomorphism from $\Gamma$ to $\msi^\times$ against this measure must be a unit in $\msi$.
Thus, recalling that the $\fp$-adic period $\Omega_\fp(A)$ is a unit in $\msi$,  it follows from equation (7.20) of \cite{CL} that
$$
\Omega_\infty^{-1}L(\ov{\phi}, 1)(1- \phi(\fp)/N\fp)
$$
will be a unit in $\msi$, and thus a unit at $\fQ$. But, noting that $\phi(\fp)\phi(\fp^*) = N\fp$, the conclusion of the proposition follows immediately
from the first assertion of Corollary \ref{4.2n}.
\end{proof}

\begin{lem}\label{4.5} For each $R \in \fR$, the extension $J_R/K$ defined by \eqref{6} is unramified at the primes of $K$ lying above 2. \end{lem}

\begin{proof} It suffices to show that, for each prime $r$ dividing $R$, the extension $K(\sqrt{r})/K$ is unramified at the primes above 2. Put $m = (\sqrt{r}-1)/2$,
so that $V(m) = 0$, where $V(X) = X^2 + X - (r-1)/4$. But then $V'(m) = 2m-1$ is a unit at $\fp$ and $\fp^*$, and so $K(m) = K(\sqrt{r})$ is unramified at the primes of $K$ above 2.
\end{proof}

\noindent

For each positive divisor $d$ of $R$, we define
\begin{equation}\label{4.6}
\msl(d) = \sqrt{d}L(\ov{\phi}_d, 1)/\Omega_\infty, \, \, \msl = \msl(1).
\end{equation}
Proposition \ref{3.3n} shows that $\msl(d)$ always belongs to the field $HT$. However, the following stronger assertion is essential for our proof. Note
that by Rohrlich's theorem \cite{Ro1}, we have $\msl \neq 0$ for all primes $q \equiv 7 \mod 8$.
\begin{thm}\label{4.7} Assume $R \in \fR$, and let $d$ be any positive divisor of $R$. Then $\msl(d)/\msl$ belongs to the field $T$.
\end{thm}
\begin{proof} For each integral ideal $\fa$ of $K$, which is prime to $R\fq$, we define
$$
e_\fa = \phi(\fa)/\xi(\fa).
$$
Then, as is already shown in \cite{GS} (see Corollary 4.11), $e_\fa$ only depends on the ideal class of $\fa$, and so, writing $\sigma$ for the Artin symbol of $\fa$
in $\fG = \Gal(H/K)$, we put $e_\sigma = e_\fa$. Now let $d$ be any positive divisor of $R$, so that $\msl(d)$ belongs to the field $HT$. Recall also that $\Gal(HT/T)$
is isomorphic $\Gal(H/K)$ under restriction, because $H \cap T = K$. If $\tau$ is any element of $\Gal(HT/T)$, we write $\tau_H$ for its restriction to $H$. Then it is proven in
Proposition 11.1 of  \cite{BG} that, for every $\tau$ in $\Gal(HT/T)$, we have
$$
\tau(\msl(d)) = e_{\tau_H} \msl(d).
$$
Since the factor $e_{\tau_H}$ is independent of $d$, it follows that $\msl(d)/\msl$ must belong to $T$, and the proof is complete.
\end{proof}
We now give the proof of Theorem \ref{AR-nonzero} using induction on the number $k$ of prime factors of $R$. Assume first that $k=1$, so that $R = r$, a prime number.
 By Theorem \ref{key-identity} and Proposition \ref{inte-2} for the the prime $r$, we conclude that
\begin{equation}\label{4.8}
\msl(r)/\sqrt{r}+ (1-\ov{\phi}((r))/r^2)\msl = 2V_r,
\end{equation}
where  $V_r = \lim_{n\ra \infty}\sum_{\fa\in\fC_n} \xi(\fa)\Psi_{\fa, r}$. Let $\fW$ be any prime of $HT(\sqrt{r})$ lying above the prime $\fP$ of $T$. By Proposition \ref{inte-2}, we have  $\ord_\fW(V_r) \geq 0.$ Further, by Lemma 4.3, we have $(1-\ov{\phi}((r))/r^2) = (1 + 1/r)$. Thus, since $r+1 \equiv 2 \mod 4$, it follows from Proposition \ref{4.4n} that $\ord_\fW((1-\ov{\phi}((r))/r^2)\msl) = 0.$ As $\ord_\fW(V_r) \geq 0$, we conclude from \eqref{4.8} that $\ord_\fW(\msl(r)/\sqrt{r}) = 0$, and so Theorem \ref{AR-nonzero} holds when   $k=1$.

\medskip

A curious new aspect of the argument now arises when $R = r_1\ldots r_k$ with $k \geq 2$, and we must appeal to Theorem \ref{4.7} to get around it. By Theorem \ref{key-identity}, we have
\begin{equation}\label{4.9}
\msl(R)/\sqrt{R} + \sum_{d|R, d \neq 1, R}\Lambda(d, R)/\sqrt{d} + \msl \prod_{i=1}^{k}(1- \ov{\phi}((r_i))/r_i^2) = 2^kV_R,
\end{equation}
where $V_R =  \lim_{n\ra \infty}\sum_{\fa\in\fC_n} \xi(\fa)\Psi_{\fa, R}$, and
$$
\Lambda(d, R) = \msl(d) \prod_{r|R/d}(1-\ov{\phi}_d((r))/r^2).
$$
The problem is that the terms $\Lambda(d, R)$ lie in an extension of $HT$ where the prime $\fP$ of $T$ is unramified but will usually have a large residue class field extension, and
this means one cannot carry through the inductive argument in its most naive form. The key to overcoming this difficulty is to divide both side of \eqref{4.9} by the non-zero number $\msl$. Doing this, and defining, for each positive integer divisor $d$ of $R$, $\Phi(d, R) = \Lambda(d, R)/\msl$, we obtain the equation
\begin{equation}\label{4.10}
\Phi(R)/\sqrt{R} + \sum_{d|R, d \neq 1, R}\Phi(d, R)/\sqrt{d}  + \prod_{i=1}^{k}(1- \ov{\phi}((r_i))/r_i^2) = 2^kV_R/\msl,
\end{equation}
where $\Phi(R) = \msl(R)/\msl$. Let $H_R$ be the field defined in \eqref{6}, and we now take $\fW$ to be any prime of the compositum $H_RT$ lying above $\fP$, so that $\fW/\fP$ is unramified. By Proposition \ref{inte-2}, we have  $\ord_\fW(V_R) \geq 0.$ Thus we conclude from Proposition \ref{4.4n} that $\ord_\fW(2^kV_R/\msl) \geq k+1$. Thanks to  Lemma \ref{4.3}, we have
\begin{equation} \label{4.11}
\ord_\fW(\prod_{i=1}^{k}(1- \ov{\phi}((r_i))/r_i^2)) = k.
\end{equation}
On the other hand, our inductive hypothesis, together with Lemma \ref{4.3} and Proposition \ref{4.4n},  shows  that, for each positive divisor $d$ of $R$, with $d \neq 1, R$, we have
\begin{equation}\label{4.12}
\ord_\fW(\Phi(d, R)/\sqrt{d}) = k.
\end{equation}
Of course, these estimates alone do not allow us to conclude from \eqref{4.10} that $\ord_\fW(\Phi(R)/\sqrt{R}) = k.$ However, the argument is saved by Theorem \ref{4.7},
which tells us that, for every positive divisor $d$ of $R$, $\Phi(d, R)$ belongs to the field $T$, and so it lies in the completion $T_\fP$ at $\fP$. Since $\fP$ has its residue
field of order 2, this means that we can write, for every positive divisor $d \neq 1, R$ of $R$,
\begin{equation}\label{4.13}
\Phi(d, R)/\sqrt{d} = \sqrt{d}\pi_\fP^k(1 + \pi_\fP b_d),
\end{equation}
where $\pi_\fP$ is a local parameter at $\fP$, and $\ord_\fP(b_d) \geq 0$. Thus
\begin{equation}\label{4.14}
\sum_{d|R, d \neq 1, R}\Phi(d, R)/\sqrt{d}  \equiv \pi_\fP^kD_R \mod \fW^{k+1},
\end{equation}
with  $D_R = \sum_{d|R. d \neq 1, R}\sqrt{d}$. But
$$
D_R^2 \equiv  \sum_{d|R. d \neq 1, R} d \mod \fW,
$$
and $\sum_{d|R. d \neq 1, R} d \equiv 2^k \mod 2$, whence $\ord_\fW(D_R) \geq 1$. Thus we have finally shown that
$$
ord_\fW(\sum_{d|R, d \neq 1, R}\Phi(d, R)/\sqrt{d} ) \geq k+1.
$$
It now follows from \eqref{4.10} and \eqref{4.11} that $\ord_\fW(\Phi(R) = k$.  Thus, again using Proposition \ref{4.4n}, we
have finally proven Theorem \ref{AR-nonzero} by induction on the number of prime factors of $R \in \fR$.

\medskip

We end with a numerical example. Take $q = 23$. Then $K = \BQ(\sqrt{-23})$ has class number $h=3$. The Hilbert class field is $H = K(\alpha)$,
where $\alpha$ satisfies the equation $\alpha^3 - \alpha - 1 = 0.$  The following  global minimal Weierstrass equation for $A/H$ is given by Gross \cite{Gross78}
$$
y^2 + \alpha^3xy + (\alpha + 2)y = x^3 + 2x^2 - (12\alpha^2 + 27\alpha + 16)x - (73\alpha^2 + 99\alpha +62).
$$
Then $\fR$ will consist of all square free positive integers $R$ such that every prime factor $r$ of $R$ satisfies $r \equiv 1 \mod 4$ and
$r$ is congruent to one of $5, 7, 10, 11, 14, 15, 17, 19, 20, 21, 22 \mod 23$. Then Theorem \ref{main} shows that, for all $R \in \fR$, we have $L(A^{(R)}/H, 1) \neq 0.$ However,
we thank A. Dabrowski for pointing out the following interesting numerical example to us. Let $\beta = \sqrt{-23}$, and define $J$ to be the elliptic curve defined
over $H$, which is the twist of $A/H$ by the quadratic extension $H(\sqrt{-\beta})/H$. Theorem 1.3 of \cite{CL} in the case $q = 23$ asserts that
  $L(J/H, 1) \neq 0$. Now take $R= 901 = 17\times53$,  so that $R \in \fR$. Let $J^{(901)}/H$ be the twist of $J/H$ by the quadratic extension
$H(\sqrt{901})/H$. Then Dabrowski's calculations show that  $L(J^{(901)}/H, 1) = 0$. Thus the obvious analogue of Theorem \ref{main} does not hold for the curve $J/H$.

\medskip

\noindent John Coates,\\
Emmanuel College, Cambridge,\\
England.\\
{\it jhc13@dpmms.cam.ac.uk }

\medskip

\noindent Yongxiong Li,\\
Yau Mathematical Sciences Center,\\
Tsinghua University, \\
Beijing, China.\\
{\it liyx\_1029@tsinghua.edu.cn}

\end{document}